\documentclass[11pt]{amsart}

\usepackage{amsfonts}
\usepackage{amsmath}
\usepackage{amssymb}
\usepackage{amscd, amsthm}
\usepackage{verbatim}
\usepackage{latexsym}
\usepackage[all,cmtip]{xy}
\usepackage{graphics,epic,eepic}

\newcommand{\N}{\mathbb{N}}

\newcommand{\Q}{\mathbb{Q}}

\newcommand{\G}{\mathbb{G}}
\newcommand{\X}{\mathcal{X}}
\newcommand{\f}{\mathcal{F}}
\newcommand{\kay}{\mathcal{K}}
\newcommand{\el}{\mathcal{L}}
\newcommand{\EA}{\mathcal{A}}

\newcommand{\Z}{\mathbb{Z}}
\newcommand{\Zp}{\mathbb{Z}_{p}}

\newcommand{\aQp}{\overline{\mathbb{Q}_{p}}}

\newcommand{\Qp}{\mathbb{Q}_{p}}
\newcommand{\Oc}{\mathcal O}
\newcommand{\GL}{\mathrm{GL}}

\newcommand{\limproj}{\underleftarrow{\lim}}

\newtheorem{theorem}{Theorem}[section]
\newtheorem{lemma}[theorem]{Lemma}
\newtheorem{prop}[theorem]{Proposition}

\newtheorem{cor}[theorem]{Corollary}

\newtheorem{conj}[theorem]{Conjecture}
\newtheorem{definition}[theorem]{Definition}

\begin{document}

\title[Iwasawa modules]
{Ramification in  Iwasawa modules}

\author[Chandrashekhar Khare]{Chandrashekhar Khare}
\email{shekhar84112@gmail.com}
\address{Department of Mathematics, UCLA \\
        Los Angeles, CA 90095-1555, U.S.A.}\thanks{CK was partially supported by NSF grants.}

\author[J-P. Wintenberger]{Jean-Pierre Wintenberger}
\email{wintenb@math.u-strasbg.fr}
\address{Universit\'e de Strasbourg \\
D\'epartement de Math\'ematique \\
Membre de l'Institut Universitaire de France\\
7, rue Ren\'e Descartes, 67084,
Strasbourg Cedex \\
France}\thanks{JPW is member of the Institut Universitaire de France.}



\maketitle


\begin{abstract}
We make a  reciprocity conjecture that extends Iwasawa's analogy of direct limits of class groups along the cyclotomic tower of a totally real number field $F$ to  torsion points of Jacobians of curves over finite fields.  The extension is to generalised class groups and generalised Jacobians. 
We  state  some ``splitting
conjectures'' which are  equivalent to Leopoldt's conjecture.
\end{abstract}


\section{Introduction}

For a number field $F$, with ring of integers $\Oc_F$, we may define the class group of $F$ to be ${\rm Pic}(\Oc_F)$, i.e., the isomorphism classes of invertible sheaves 
on ${\rm Spec}(\Oc_F)$.  Iwasawa deepened this formal analogy between class groups of number fields and Jacobians. He considered   ${\mathcal X}_\infty^-$, the  inverse limit under norm maps of the minus parts under complex conjugation of the Sylow $p$-sugroups of  the class groups of  $F(\mu_{p^n})$, where $F$ is a totally real number field, $p$ a fixed (odd) prime, and  $n$ varying. Iwasawa viewed
${\mathcal X}_\infty^- \otimes \Q_p$ as a $p$-adic vector space, which he proved to be finite dimensional, equipped  with the action of $\gamma$, a generator for the $p$-part of ${\rm Gal}(F(\mu_{p^\infty})/F)$. He conjectured that the characteristic polynomial  for this action should be the same as a certain $p$-adic $L$-function, at least when $F=\Q$. This was later called the main conjecture of Iwasawa theory which was proved by Mazur-Wiles (for $F=\Q$) and Wiles (for general totally real $F$). Iwasawa's conjecture can be viewed as an analog of the theorem of Weil which relates zeta-functions 
of curves over finite fields of characteristic $p$, to the characteristic polynomial for the action of Frobenius on the $\ell$-adic Tate module of its Jacobian, for $\ell \neq p$.

In this paper we ask for  an Iwasawa theoretic analog of a standard fact in the theory of generalised Jacobians, that holds over arbitrary base fields and is easier  than Weil's result mentioned above.  Namely,  let $X$ be a smooth projective curve over  a field $K$ with Jacobian $J$.  We have  the isomorphism ${\rm Ext}^1(J,\G_m)={\rm Pic}^0(J)=J$.  Let  $P,Q \in X(K)$ be an ordered pair of distinct points, and   consider the {\it  generalised Jacobian} $J_{P,Q}$,  the Jacobian of the singular curve $X'$  obtained from $X$  by identifying $P$ with $Q$. Thus $X'$ is a curve over $K$  with nodal singularity.  We have an exact sequence $$0 \rightarrow \G_m \rightarrow J_{P,Q} \rightarrow J \rightarrow 0.$$ The standard fact alluded to earlier is that the  class of $J_{P,Q}$ in ${\rm Ext}^1(J,\G_m)$ is given by the class of the degree 0 divisor  $(P)-(Q)$.  
We make a  reciprocity conjecture, see Conjecture \ref{rec}, that asks for an analogous formula in Iwasawa theory. To formulate this conjecture, we consider ramification at {\it auxiliary primes} in Iwasawa modules (see \S \ref{odd}),  define analogs of degree
0 divisors supported on Frobenius elements in certain Galois groups (see \S \ref{degree}), and use a well-known pairing of Iwasawa (see \S \ref{Iwasawa}). We prove an implication of  the reciprocity conjecture  (see Theorem \ref{intersection} and Corollary   \ref{evidence}). The proof of the reciprocity conjecture has eluded us. 

If  the field $K$ above is a finite field, then the extension class $(P)-(Q)$ is  of finite order. Inspired by Iwasawa's analogy, we conjecture in our situation too that the extension classes in the reciprocity conjecture  are of finite order. This leads to a splitting conjecture, see  Conjecture \ref{rational1}, that we show in Corollary \ref{implication1} to be  equivalent to the following standard  conjecture:

\begin{conj} \label{Leopoldt}(Leopoldt) 
The cyclotomic $\Z_p$-extension  $F_\infty/F$  is the unique $\Z_p$-extension of  a totally real number field $F$.
\end{conj}

We denote by $\delta_{F,p}$, the integer such that the $\Z_p$-rank of the maximal abelian $p$-extension of $F$ unramified outside $p$ is $1+\delta_{F,p}$. The conjecture asserts that it is 0; $\delta_{F,p}$ is also called the Leopoldt defect (for $F$ and $p$).

Our original motivation for this work was to search for  a criterion for Leopoldt's conjecture that could be approached using Wiles' proof of the main conjecture \cite{Wiles} which draws on Hida's  theory of  $\Lambda$-adic Hilbert modular  forms.  This search led to Conjecture \ref{rational1}. As Conjecture \ref{rational1} is about odd extensions of $\f_\infty$, it  might offer some access to methods that use Hilbert modular forms.

\subsection{\it Notation}\label{notation} 
 We fix a prime number $p$ throughout. Except in paragraph \ref{kummertheory}, we make the assumption that $p$ is odd. 
We let $F$ be a totally real number field. We operate within a fixed algebraic closure $\overline F$ of $F$.   We have the cyclotomic $\Gamma(=\Z_p)$-extension of $F$  that we denote by $F_\infty$.  We denote by $\gamma$ a chosen topological generator of $\Gamma$, and by $\chi$ the $p$-adic cyclotomic character.  The field $F_\infty$  is contained in ${\f}_\infty=F(\mu_{p^\infty})$, whose  real subfield we denote by  $F^\infty$; $F_\infty$ is contained in $F^\infty$. The degree $[{\f}_\infty:F_\infty]$ divides $p-1$ and $[\f_\infty:F^\infty]=2$.   We denote by $\f_n$ and $F_n$ the extension $F(\mu_{p^{n+t}})$ and its real subfield respectively. Here $t$ is the largest integer so  that $F(\mu_p)$ contains the $\mu_{p^t}$ roots of unity. Hence $[\f_n : F(\mu_p )]= [F_n :F]=p^n$.
For convenience we will  assume throughout the paper that $F_\infty=F^\infty$, i.e., $[F(\mu_p):F]=2$.  
For a finite place $q$ of a number field $F$ we denote by $N(q)$ its norm,  the order of the residue field at $q$.  For a finite set of finite places $Q$ of $F$, by the $Q$-units of $F$,  denoted by $E_Q$, we mean elements of $F^*$ which are units at  all finite places outside $Q$. 

For an abelian group $M$, we denote by $\widehat M$ its  prop-$p$ completion $\limproj_n M/M^{p^n}$.

We say that an  abelian extension $L$ of $\f_\infty$ is odd (or its Galois group is odd) if
$L$ is Galois over $F$ and the complex conjugation of ${\rm Gal}(\f_\infty/F)$ acts on ${\rm Gal}(L/\f_\infty)$ by  inversion.


 By the $\Z_p$-rank of an $\Z_p$-module $M$,  called the essential rank by Iwasawa, we mean the dimension of $M \otimes \Q_p$ as a vector space over $\Q_p$. For a $\Lambda=\Z_p[[T]]=\Z_p[[\Gamma]]$-module $M$, and an integer $n$, we denote by $M(n)$ the $\Lambda$-module with same underlying module $M$, and the $\Lambda$-action specified by $\gamma.m=\chi(\gamma)^n \gamma m$. We say that (possibly infinite) Galois extensions $L,L'$ of a field $K$ are almost linearly disjoint if the degree $[L \cap L':K]$ is finite.  Given a Galois  extension  $L/K$ of algebraic (possibly infinite)  extensions of $\Q$, we may talk about places of $K$ and  conjugacy class of decomposition groups, inertia groups at these places. If $L/K$ has abelian Galois group  we say that $L/K$ is almost totally ramified at a set of places of $K$ if the  inertia groups at these places generate a subgroup of finite index of ${\rm Gal}(L/K)$.
 

 \subsection{\it Acknowledgements}
 
 We would like to thank Gebhard B\"ockle,   John Coates,  Najmuddin Fakhruddin,   David Gieseker, Ralph Greenberg, Benedict Gross, 
 Haruzo Hida,  Tony Scholl, Chris Skinner, Kevin Ventullo for helpful conversations.  
  The first author thanks the D\'epartement de Mathematiques of the Universit\'e de Strasbourg
  for its support during a visit in the summer of 2009 when  some of the work reported on in this paper was done.
  
  Part of the writing of  this work was done during the authors' stay at the Institut Henri Poincare - Centre Emile Borel and IAS, Princeton. The authors thank these institutions for hospitality and support. 

  \section{Some Kummer theory}\label{kummertheory}
 
 In this section, we state some results on Kummer theory 
 and  $\Zp$-extensions  which presumably are well known.  They are basic to the work 
 of this paper. For lack of a reference known to us, we provide proofs of these results.

Let $p$ be any prime number for this section, allowing $p=2$.

\subsection{General fields}\label{kummergen} 

Let  $F$ be any field
of characteristic different from  $p$. Recall that $\f_\infty $ is the 
cyclotomic extension $F(\mu_{p^\infty})$.
Let $L$ be an extension of $\mathcal{F}_\infty$. We say that $L$ is \it a Kummer 
$\Zp$-extension \rm of $F$  if $L/F$ is Galois and it is such that ${\rm Gal}(L/\f_\infty)\simeq \Z_p$ is isomorphic to $\Z_p(1)$ as a ${\rm Gal}(\f_\infty/F)$-module.
 We let $\widehat{F^*} $ be the $p$-adic completion 
 of the multiplicative groupe of $F$ \it i.e. \rm the projective limit $\limproj_n  F^*/(F^*)^{p^n}$, the 
 transition maps being induced by the identity. 
 
 We have the Kummer isomorphisms  $\mathrm{K}_{F,n}: F^* /(F^*)^{p^n}  \rightarrow 
H^1 (G_F , \mu _{p^n})$. Taking  the projective limits for $n$, 
we get an isomorphism   $\mathrm{K}_F : \widehat{F^*}\rightarrow H^1 (G_F ,\Zp (1))$, where 
the $H^1$ are continuous $H^1$, the topology 
of $\Zp (1)$ being the $p$-adic one (\cite{NSW}).

If $\bar{x} =(\bar{x}_n )_{n\in \N}$ is an element of $\widehat{F^*}$, we note 
$F_{\bar{x}}$ the extension 
of $\mathcal{F}_\infty$ which is the union of the Kummer extensions
$F(\mu _{p^n}, x_n ^{1/p^n})$, where $x_n \in F^*$ maps
to $\bar{x}_n $ in $F^*/(F^*)^{p^n}$. It is also the extension
of $\f _\infty$ corresponding to the fixed field of the kernel 
of the homomorphism arising from the image of $\mathrm{K}_F (\bar{x}) $  under  the map 
$H^1(G_F , \Zp (1))\rightarrow \mathrm{Hom}(G_{\mathcal{F}_\infty}, \Zp )(1)^0,$
where the $\mathrm{Hom}$
are continuous homomorphisms and  $^0$ means fixed by $\mathrm{Gal}(\mathcal{F}_\infty / F)$.   
 
 For a subgroup $T$ of $\widehat{F^*}$, by $F(\mu_{p^{\infty}},T^{1 \over p^\infty})$ we mean
 the compositum of all extensions of $F$  obtained by adjoining,  for all $n \in \N$,  all $p^n$ th roots of (lifts to $F^*$ of) the image of $T$ in $F^*/{(F^*)}^{p^n}$ : it is the union of 
 the fields $F_{\bar{x}}$ for $x\in T$. If $T$ is a subgroup of $F^*$ we still 
 denote $F(\mu_{p^{\infty}},T^{1 \over p^\infty})$ the extension defined by the image of 
 $T$ in $\widehat{F^*}$.

 \begin{prop}\label{kummerprop}
 The Kummer $\Zp$-extensions of $\f _\infty$ are exactly the fields
 $F_{\bar{x}}$, for $\bar{x}\in \widehat{F^*}$ non-torsion. The torsion 
 of  $\widehat{F^*}$  is the group $\mu_{p^\infty }(F)$  of roots of unity of order a power of $p$
 if this group is finite,  and is trivial if $F=\f _\infty$.
 \end{prop}
 
 \begin{proof}
 Let $\bar{x}\in \widehat{F^*}$ be such that $ \bar{x}^{p^a}=1$. 
 Write $\bar{x}= (\bar{x}_n)_n$ with $x_n \in F^*$. For every $n$, there exists
 $y_n \in F^*$ such that $x_n ^{p^a}= y_n^{p^n}$. For $n\geq a$, it follows that 
 $\epsilon _{n-a} := x_n y_n^{-p^{n-a}}$ is a $p^a$ root of unity. 
 We have $(\bar{x}_n )=(\bar{\epsilon }_n)$. If  $\mu_{p^\infty }(F)$ is finite, 
 it follows that there 
 exists an $\epsilon \in \mu_{p^\infty }(F)$  such that 
the $\bar{\epsilon} _n$ for $n\in \N$ are the image of $\epsilon$.
If $\mu_{p^\infty }(F)$ is infinite, it is $p$-divisible, and it follows that the torsion 
of $\widehat{F^*}$ is trivial. This proves the part of the proposition
concerning the torsion of $\widehat{F^*}$.
 
 If $\mu_{p^\infty }(F)$ is infinite, the proposition follows from
 the fact that the Kummer map $\mathrm{K}_F$ is bijective.
  Let us suppose that $\mu_{p^\infty }(F)$ is finite.
 
 \begin{lemma} The cohomology groups 
 $H^1 (\mathrm{Gal}(\f _\infty /F), \mu_{p^n} (\overline{F}))$
 and $H^2 (\mathrm{Gal}(\f _\infty /F), \mu_{p^n} (\overline{F}))$
 are killed by 
 a power $p^a$ of $p$ independent of $n$. \end{lemma} 

Let us prove the proposition granted the lemma. As the projective system
$\Zp (1) = \limproj_n \mu _{p^n} (\overline{F} )$ satisfies the Mittag-Leffler property,
and the functor projective limit is left exact, Hochschild-Serre exact sequences 
for coefficients $\mu_{p^n} (\overline{F} )$
give the following exact sequence:
 $$(0) \rightarrow H^1(\mathrm{Gal}(\f_\infty / F), \Zp (1))\rightarrow
H^1 (G_F, \Zp (1))\rightarrow H^1 (G_{\f_\infty}, \Zp (1)),$$
and the $H^1$ with coefficients in $\Zp (1)$ are the projective limit 
of the $H^1$ with coefficients in $\mu_{p^n} (\overline{F})$
(use cor. 2.7.6. of chap. 2 paragraph 7 of \cite{NSW}).
The lemma implies that  $H^1(\mathrm{Gal}(\f_\infty / F), \Zp (1))$
is torsion.
It then follows from the above  exact sequence, the fact that 
 $H^1 (G_{\f_\infty}, \Zp (1))= \mathrm{Hom}(G_{\f_\infty} , \Zp (1))$ has no torsion, and the bijectivity of 
the Kummer map $\mathrm{K}_F$, that the kernel 
of the map 
$\widehat{F^*} \rightarrow \mathrm{Hom}(G_{\f_\infty} , \Zp (1))$ is the torsion subgroup
of $\widehat{F^*}$. It follows that if $\bar{x}$ is not torsion, the extension 
$F_{\bar{x}}$ is a $\Zp$ Kummer extension of $\f _\infty$.

Conversely, let $L$ be a Kummer $\Zp$-extension of $\f_\infty$. Let $f$ be a continuous 
non zero morphism $G_{\f _\infty } \rightarrow \Z_p (1)$ whose kernel corresponds to $L$. 
Let $f_n $ be the morphisms $G_{\f _\infty } \rightarrow \mu _{p ^n} (\overline{F})$ 
defined by $f$. As  $H^2 (\mathrm{Gal}(\f _\infty /F), \mu_{p^n} (\overline{F}))$ is killed
by $p^a$, $p^a f_n$ is the image of an element $\bar{x}_n$ of 
$F^* /(F^*)^{p^n}$. As  $H^1(\mathrm{Gal}(\f _\infty /F), \mu_{p^n} (\overline{F}))$ is killed
by $p^a$, the $\bar{x}_n ^{p^a}$ define an element  $\bar{x}'$ in the projective limit
$\limproj_n F^*/(F^*)^{p^n}$, hence of $\widehat{F^*}$. One has 
$\mathrm{K}_F (\bar{x}')=p^{2a} f$, hence $L= F_{\bar{x}'}$. This  proves the proposition, granted the lemma.

Let us prove the lemma. Let $F'$ be $F(\mu_p (\overline{F}))$ if $p\not= 2$ and $F(\mu_4(\overline{F}))$ if 
$p=2$. 
By Hochschild-Serre spectral sequence , we reduce to the case $F= F'$.
Note that if  $\mu_{p^\infty} (F)$
is infinite, the lemma is obvious as $\f _\infty = F$. 
So we may suppose that $\mathrm{Gal} (\f _\infty /F)$ is isomorphic to $\Zp$. 
Let $\gamma$ a generator of $\mathrm{Gal} (\f _\infty /F)$ and 
$\chi_p (\gamma)$ its image by the cyclotomic character.
The calculation of the cohomology
of the procyclic group $\Zp$ gives that  $H^1 (\mathrm{Gal}(\f _\infty /F), \mu_{p^n} (\overline{F}))$ 
is isomorphic to $(\Z /p^n \Z)/ (\chi_p (\sigma)-1)$ and  
$H^2 (\mathrm{Gal}(\f _\infty /F), \mu_{p^n} (\overline{F}))$ is trivial (prop. 1.7.7
of chap. 1 paragraph 7 of \cite{NSW}). The lemma follows as 
$\chi _p (\gamma)\not= 1$. 
\end{proof}

\it Remarks. \rm It follows from the proof of the proposition that $\widehat{F^*}$ injects in 
 $H^1 (G_{\f_\infty}, \Zp (1))$. It implies the following.
Let $\bar{x}_i$, $i=1,2$, be two non-torsion elments of 
$\widehat{F^*}$. Then $F_{\bar{x}_1}= F_{\bar{x}_2}$ if and only 
if there exist $a_1$ and $a_2$ in $\Zp$, non-zero, such that 
$\bar{x}_1^{a_1}= \bar{x}_2 ^{a_2}$.   

The proof of the proposition implies that if $T$ is a finitely 
generated subgroup of $F^*$, the Galois group of $F_T=F(\mu_{p^{\infty}},T^{1 \over p^\infty})$
over $\mathcal{F}_\infty = F(\mu_{p^{\infty}})$ is a finitely generated 
abelian group of  the same $\Zp$-rank as the closure of $T$ in $\widehat{F^*}$.

\subsection{Number fields}\label{kummernf}

We suppose now that $F$ 
is a finite extension of $\Q$. If $\mathfrak{q}$ is a prime of $F$,
we denote by $F_{\mathfrak{q}}$ the completion of $F$ at $\mathfrak{q}$. We denote by 
$v_{\mathfrak{q}}$ the 
valuation of $F_{\mathfrak{q}}$ normalized by $v_{\mathfrak{q}} (F_{\mathfrak{q}} ^* )= \Z$. We still denote by $v_{\mathfrak{q}}$ the map $\widehat{F_{\mathfrak{q}} ^*}\rightarrow \Zp$ 
induced by $v_{\mathfrak{q}}$. We denote by $\mathrm{loc}_{\mathfrak{q}}$
the morphism $\widehat {F^*} \rightarrow \widehat{F_{\mathfrak{q}}^*}$ induced 
by the inclusion of $F$ in $F_{\mathfrak{q}}$. 
\begin{prop}\label{ramification} Let $\bar{x}\in \widehat{F^*}$ be non-torsion. 
Then, the Kummer extension $F_{\bar{x}}/ \f _\infty$ is unramified at primes 
above $\mathfrak{q}$ if and only if $\mathrm{loc}_{\mathfrak{q}}(\bar{x})$ is torsion.
\end{prop}
\begin{proof} Let us note $E= F_{\mathfrak{q}}$ and  $E_{\mathrm{ur}}$  the maximal
unramified extension of $E$. The proposition follows from proposition
\ref{kummerprop} and the fact that the kernel of $\widehat{E^*}
\rightarrow \widehat{ E_{\mathrm{ur}}^*}$ is torsion. For this fact, 
let $\omega$ be a uniformizer of $E$. If $\mathfrak{q}$
is not above $p$, we have 
$\widehat{E^*}\simeq \omega^{\Zp} \mu_{p^\infty}(E)$ and
$\widehat{E_{\mathrm{ur}^*}}\simeq \omega^{\Zp}$. 
If $\mathfrak{q}$ is above $p$, we have $\widehat{E^*}\simeq 
\omega^{\Zp}U^+_E$ and  $\widehat{E_{\mathrm{ur}}^*}\simeq 
\omega^{\Zp}U^+_{\widehat{E_{\mathrm{ur}}}}$, where $U^+$ are 
units that $\equiv 1 \ \mathrm{mod} . \omega$ and $\widehat{E_{\mathrm{ur}}}$
is the completion of $E_{\mathrm{ur}}$.
The map $U^+ _F \rightarrow U^+_{\widehat{E_{\mathrm{ur}}}}$,
is injective as  $U^+_{E_{\mathrm{ur}}}$ is separated 
for the $p$-adic topology. \end{proof}

\it Remark. \rm  
The proof of the proposition shows that  if $F_{\bar{x}}/F$ is unramified 
at $\mathfrak{q}$, $v_{\mathfrak{q}}(\bar{x})=0$, the converse being 
true if  $\mathfrak{q}$ is not above $p$.


\vskip .3cm

We now let $Q$ be a finite set of primes of $F$. We denote by $E_Q$ the $Q$-units 
\it i.e. \rm the  elements $x\in F^*$ such that $v_{\mathfrak{q}}(x)=0$
for $\mathfrak{q}\notin Q$. 
The group $E_Q$ is finitely generated.  We write $\widehat{E_Q}$
its $p$-adic completion. 
As if a power of $x\in F^*$
is a $Q$-unit, then $x$ is a $Q$-unit, the natural maps
$E_Q /E_Q ^{p^n}\rightarrow F^* /(F^* )^{p^n}$ are injective, hence also the map
$\widehat{E_Q}
\rightarrow \widehat{F^*}$. We  identify $\widehat{E_Q}$
to a subgroup of $\widehat{F^*}$.  

\begin{prop}\label{Q} a) An element $\bar{x}\in \widehat{F^*}$ belongs to $\widehat{E_Q}$
if and only if $v_{\mathfrak{q}} (\bar{x})=0$ for $\mathfrak{q}\notin Q$. 

b) If 
$\bar{x}$ is non-torsion, the Kummer $\Zp$-extension $F_{\bar{x}}/\f_\infty $ is unramified outside $Q$
only if $\bar{x}\in \widehat{E_Q}$. If the primes 
of $F$ above $p$ are
in $Q$, the converse is true. 
\end{prop}
\begin{proof} The second part of the proposition follows from the first one, 
the preceding proposition and the  remark after the proposition \ref{ramification}. 

Let us prove the first part.
The ``only if'' part is clear so let us prove the ``if'' part. 

Let $p^a$ be a power of $p$ that kills the $p$-primary part 
of the class group of the ring $O_Q$ of $Q$ integers 
(elements $x\in F$ such that $v_{\mathfrak{q}}(x)\geq 0 $ for $\mathfrak{q}\notin Q$).

Let $x=(\bar{x}_n )$ be in $\widehat{F^*}$ such that $v_{\mathfrak{q}}(x)=0$
if $\mathfrak{q}\notin Q$. 
Let $x_n\in F^*$ be a lift $\bar{x_n}$. Let $I(x_n )$ be the rank one projective $O_Q$-module
generated by $x_n$.  
As $v_{\mathfrak{q}} (x_n )$ is divisibleby $p^n$ for $\mathfrak{q}\notin Q$, 
there is rank one projective $O_Q$-module $I_n $ such that $I(x_n )= I_n ^{p^n}$.
The rank one module $I_n ^{p^a}$ is free. Let $y_n\in O_Q$
be a generator. We have 
$I(x_n )=I(y_n)^{p^{n-a}}$, hence there is $\epsilon_n$ a unit in $O_Q$  such that 
$x_n =y_n ^{p^{n-a}}\epsilon_n$. We see that $x_n$ and $\epsilon _n$
have the same image in $F^*/ (F^* )^{p^{n-a}}$. 
It follows that the $\epsilon_n$ define an element $\epsilon$ of $\widehat{E_Q}$
with image $x$ in  $\widehat{F^*}$. The proposition is proved. 
\end{proof} 

We will need the following lemma: 

\begin{lemma}\label{rami}  Let $T$ a finitely generated  subgroup   of  $F^*$, and let $Q$ be a finite set of finite places of $F$. Let $F_T=F(\mu_{p^{\infty}},T^{1 \over p^\infty})$  be the compositum of the extensions $F_t$ for $t\in T$. 
Then the $\Z_p$-rank
of the subgroup generated by the inertia groups above $Q$ in ${\rm Gal}(F_T/\f_\infty)$ is the same as the $\Z_p$-rank of the closure of  (the diagonal image of)  $T$ in $\Pi_{v \in Q} \widehat{F_v^*}$.
 \end{lemma}
\begin{proof} For $v\in Q$, let $v'$ be a prime of $\f _\infty$
above $v$ and let $I_{v'} $ be the inertia subgroup of $\mathrm{Gal}(F_T /\f _{\infty})$
at $v'$. 
As the action of $\mathrm{Gal} (\f_\infty / F)$ on 
$\mathrm{Gal} (F_T /\f_\infty )$ is by the cyclotomic character $\chi _p$,
one easily sees that the subgroup $I_{v'}$ does not depend of $v'$
and we call it $I_v$.
We have the following commutative diagram:
$$\begin{array}{ccc}
T & \rightarrow & \mathrm{Hom}(\mathrm{Gal}(F_T /\f _{\infty}), \Zp (1)) \\
\downarrow & & \downarrow\\	
\prod_{v\in Q} \widehat{F_v ^* }& \rightarrow & \prod_{v\in Q} 
 \mathrm{Hom}(I_v, \Zp (1)).
\end{array}$$
The lemma follows from the fact that the horizontal arrows have torsion
kernels by propositions \ref{kummerprop} and \ref{ramification}. 
 
\end{proof}

 \section{Elements of Iwasawa  theory}\label{Iwasawa}

Let ${\el}_\infty$ be the maximal abelian $p$-extension of ${\f}_\infty$ that is unramified everywhere. We set ${\X}_\infty={\rm Gal}(  {\el}_\infty/{\f}_\infty)$. It decomposes as
${\X}_\infty=\X_\infty^+ \oplus \X_\infty^-$ under the action of complex conjugation which corresponds to
$\el_\infty$ being the compositum of two linearly disjoint extensions $\el_\infty^+$ and $\el_\infty^-$.   The Galois group $\X_\infty$ (respectively $\X_\infty^+,\X_\infty^-$)  is the inverse limit of the $p$-parts of the class groups, denoted by $\EA_n$,  of $\f_n$  (resp., $+$ and $-$ parts, $\EA_n^+$ and $\EA_n^-$) ($n \geq 0$) under the norm maps. It is conjectured by Greenberg that $\X_\infty^+$ is a finite group. We have the theorem of Iwasawa that under the natural Galois action of $\Lambda=\Z_p[[T]]$, ${\X}_\infty$ is a finitely generated  torsion $\Lambda$-module.
 
 Let ${M}_\infty$ be the maximal abelian $p$-extension of $F_\infty$ that is unramified outside $p$. We set ${Y}_\infty={\rm Gal}(  {M}_\infty/F_\infty)$.  Again by a theorem of Iwasawa, $Y_\infty$ is a finitely generated  torsion $\Lambda$-module. (It is a consequence of the ``weak Leopoldt conjecture'' that he proved.)  We denote by $Y'_\infty={\rm Gal}(M_\infty/F)$, which sits inside an exact sequence
 \begin{equation}\label{degree1}
0 \rightarrow Y_\infty  \rightarrow Y_\infty'  \rightarrow \Z_p \rightarrow 0.\end{equation} 
We call the last map the degree map. Thus $Y_\infty$ is the   $\Z_p$-submodule  of  $Y'_\infty$ of elements of degree $0$. 

Recall  a couple of  facts:

\begin{itemize}

\item $Y_\infty,\X^-_\infty$  have no non-zero finite $\Lambda$-modules (cf. Propositions 15.36 and  13.28  of \cite{Washington}). This may also be deduced from 11.4.4 of \cite{NSW}  which states that $\X^-_\infty$ 
is the adjoint of a  finitely generated torsion $\Lambda$-module, and th. 11.4.8 of \cite{NSW}.

\item $Y_\infty \otimes \Q_p$ and $\X_\infty^- \otimes \Q_p$  are finite dimensional $\Q_p$-vector spaces.  The $\mu$-invariant of $F_\infty$ is not  known to be zero, and thus we do not know if $Y_\infty$  is a finitely generated $\Z_p$-module. The fact that $Y_\infty$ is a finitely generated $\Lambda$-module, implies that $Y_\infty/(\gamma-1)Y_\infty$ is a finitely generated $\Z_p$-module.

\end{itemize}

\subsection{\it Iwasawa involution and adjoints} For a $\Lambda$-module $X$ we denote by $X^0$ (Iwasawa dual)  the module whose underlying module is the same but where the $\Lambda$ action, denoted  by $.$  is defined by $f(T).x=f((1+T)^{-1}-1)x$ with the action on the right the original action. (This corresponds to defining the new $\Gamma$-action to be $\gamma.x=\gamma^{-1}x$).  It gives an involution on the category of Iwasawa modules. 
For a discrete $\Lambda$-module $M$, we endow ${\rm Hom}_{\Z_p}(M,\Q_p/\Z_p)$ with the $\Lambda$-action defined by
$ \gamma f(m)=f(\gamma^{-1}m)$. More generally for $\Gamma$-modules, either discrete or compact,  $M,N$, we endow  ${\rm Hom}_{\Z_p}(M,N)$  the group of continuous  $\Z_p$-linear homomorphisms  with the $\Lambda$-module structure given  by $\gamma f(m)=\gamma f(\gamma^{-1} m)$.
  \begin{lemma}\label{opposite}
   For a $\Lambda$-module $M$, such that $ M \otimes \Q_p$  is a finite dimensional vector space, we have a non-canonical  $\Lambda \otimes \Q_p$-isomorphism ${\rm Hom}_{\Z_p}(M,\Q_p)=M^0 \otimes \Q_p$.
   
  \end{lemma}
  
  \begin{proof}
   This  follows from the elementary fact  that over a field $K$, a  matrix $\in M_n(K)$ and its transpose are conjugate under the action of $\GL_n(K)$.
  \end{proof}
  
 We denote by $\tilde{\alpha}(X)$ the  Iwasawa dual of  the adjoint of a finitely generated, torsion $\Lambda$-module $X$, see \S 1 of  article 52  of \cite{Iwasawa}, or \S 15.5 of \cite{Washington}.  Thus the adjoint $\alpha(X)=\tilde \alpha(X)^0$.
 
 \begin{lemma}\label{adjoints}(Iwasawa)
 We have that $X$ and $\alpha (X)$ are pseudo-isomorphic, hence 
 $\tilde{\alpha}(X) $ is pseudo-isomorphic  to $X^0$. 
 \end{lemma}
  
\subsection{\it Iwasawa pairing}

The following basic theorem of Iwasawa and Coates  is important for us.

\begin{theorem}\label{iwa} (Iwasawa)
(i) We have a  perfect, $\Gamma$-equivariant, $\Z_p$-linear pairing 
$$Y_\infty \times \EA_\infty^- \rightarrow \Q_p/\Z_p (1),$$  which we call the Iwasawa pairing, equivalently $$Y_\infty={\rm Hom}_{\Z_p}(\EA_\infty^-,\Q_p/\Z_p(1)),$$ which we call the Iwasawa isomorphism. 
 
 (ii) We have that  $\tilde{\alpha}(\X_\infty^-)$ is pseudo-isomorphic to ${\rm Hom}(\EA_\infty^-,\Q_p/\Z_p)$. 
 
 (iii)
We have  a natural $\Gamma$-equivariant, $\Q_p$-linear  perfect pairing  $$(Y_\infty \otimes \Q_p )\times (\X_\infty^- \otimes \Q_p)  \rightarrow \Q_p (1),$$  or equivalently $$Y_\infty \otimes \Q_p={\rm Hom}_{\Lambda \otimes \Q_p}(\X_\infty^-\otimes \Q_p,\Q_p(1)).$$ 
\end{theorem}

\begin{proof}

(i) This is in \cite{Iwasawa} and  \cite{Coates} (see also  Proposition 13.32 of \cite{Washington} or Theorem 11.4.3 of \cite{NSW}).

(ii) Proposition 15.34 of \cite{Washington} and its proof, or Theorem 11.1.8 of \cite{NSW}.

(iii) 
Theorem  11.1.8 of \cite{NSW}      gives an isomorphism of 
$Y_\infty={\rm Hom}_{\Z_p}(\EA_\infty^-,\Q_p/\Z_p(1))$
to $\alpha (\X_\infty ')(1)$
where $\X_\infty '$ is a sub $\Lambda $-module of    $\X_\infty$ of finite 
index. The natural map 
$\alpha (\X _\infty )(1)\rightarrow \alpha(\X _\infty ')(1)$ is an isomorphism after $\otimes \Qp $.
It is the same for the natural map  $\alpha(\X_\infty  /\{p^\infty-\mathrm{torsion}\})(1)\rightarrow
\alpha (\X_\infty )(1)$. 
As  for $X$ finitely generated torsion $\Lambda$-module
without $p$-torsion, $\alpha (X)$ is isomorphic to $\mathrm{Hom}_{\Zp} (X, {\Zp })$
(corollary 1.5.7. of \cite{NSW}), we get an isomorphism 
of $\alpha(\X_\infty  /p^*-\mathrm{tors})(1)$ to $\mathrm{Hom}_{\Zp} (\X _\infty , \Zp (1))$.  

\end{proof}

\vskip .5cm


\vskip .5cm

\section{Degree 0 divisors on Frobenius elements}\label{degree}

  We observe that $(\gamma-1)Y_\infty$ is the closed commutator subgroup of $Y'_\infty$. Thus  as $Y'_\infty={\rm Gal}(M_\infty/F)$ and $M_\infty$ is ramified only at the places above $p$, for each finite 
 place  $q$ of $F$  away from  $p$ we can consider the {\it Frobenius element} ${\rm Frob}_q$ of $Y'_\infty/(\gamma-1)Y_\infty$. As no prime $q$ of $F$ is fully decomposed in the cyclotomic
 extension $F_\infty /F $, we see that $\mathrm{deg}(\mathrm{Frob}_q )\not= 0$
 for every $q$.

We have an  exact sequence  of $\Z_p$-modules deduced from (\ref{degree1}) that will also be of importance to us:
\begin{equation}\label{degree2} 0 \rightarrow Y_\infty /(\gamma-1)Y_\infty  \rightarrow Y_\infty'/(\gamma-1)Y_\infty  \rightarrow \Z_p \rightarrow 0.\end{equation}
We consider a finite set of finite places $Q=\{q\}$ of $F$ away from $p$, and thus unramified in $M_\infty/F$. 

\begin{definition}
Let $M'_Q$ be the $\Z_p$-submodule of $Y'_\infty/(\gamma-1)Y_\infty$ generated by the ${\rm Frob}_q$'s for $q \in Q$, and $M_Q$ the $\Z_p$-submodule of $M'_Q$  that is mapped to 0  under the map $Y_\infty'/(\gamma-1)Y_\infty  \rightarrow \Z_p$ of (\ref{degree2}).
We call $M_Q$ the (degree 0) Frobenius module (attached to $Q$).

\end{definition}

\begin{lemma}\label{trivial}
$M_Q$ is the $\Z_p$-span of the  degree $0$,  $\Z_p$-submodules $M_{q,q'}$ generated by
${\rm Frob}_q, {\rm Frob}_{q'}$ for $q,q' \in Q$, where in fact we may hold a $q' \in Q$ fixed as long as  the subgroup generated by the image of ${\rm Frob}_{q'}$ in $\Gamma$ contains
that generated by ${\rm Frob}_q$ for all $q \in Q$ .
\end{lemma}

\begin{proof}
 Note that the image of ${\rm Frob}_q$  in $\Gamma$  of   (at least) one element  $q \in Q$  generates 
 the   subgroup of  $\Gamma$ generated by the ${\rm Frob}_q$'s for $q \in Q$.   We choose one such and call it $q'$. Thus if we have an element $\alpha=\sum_{q \in Q } a_q {\rm Frob}_q \in M_Q,a_q\in \Z_p,$ of degree 0, 
  we can rewrite $\alpha$ as 
$\sum_{q \in Q \backslash \{q'\}} (a_q {\rm Frob}_q-a_{q,q'} {\rm Frob}_{q'})$ for some $a_{q,q'} \in \Z_p$ such that the degree of $a_q {\rm Frob}_q-a_{q,q'} {\rm Frob}_{q'}$ is 0.  
\end{proof}


 We will need to consider in the applications more particular choices of the set $Q$. 
 
 \begin{prop}\label{cebotarev}
There is a finite  set of primes $Q=\{q\}$ of $F$ away from $p$   such that ${\rm Frob}_q$'s for $q \in Q$ topologically generate  $Y'_\infty/(\gamma-1)Y_\infty$.  For such $Q$, $M_Q$ equals $Y_\infty/(\gamma-1)Y_\infty$. We may further impose  that  the image  of the ${\rm Frob}_q$ in $\Gamma$ is a generator for all $q\in Q$.
\end{prop}
\begin{proof}
It is enough to choose a finite set of $q$'s so that the  ${\rm Frob}_q$'s generate the (finite extension given by the)  maximal abelian $(p,\cdots,p)$ extension of $F$ that is unramified outside $p$, and such that $q$ is inert  in $F_\infty/F$.  By Burnside's theorem such ${\rm Frob}_q$'s generate $Y'_\infty/(\gamma-1)Y_\infty$.  
The $\Z_p$-module $M_Q$  of degree 0 is by our choice  all of   $Y_\infty/(\gamma-1)Y_\infty$, the degree 0 submodule of $Y'_\infty/(\gamma-1)Y_\infty$.
\end{proof}

In the next lemma, we consider compact groups  $M$ with a continuous action 
of $\Gamma$ that comes from a structure of $\Lambda$-module of finite type on $M$,
and the topology on $M$ is the $\mathfrak{m}_\Lambda $-topology, where 
$\mathfrak{m}_\Lambda$ is the maximal ideal of $\Lambda$.
Equivalently, $M$ is the projective limit of  a projective system of finite p-groups $M_n$
with compatible actions of $\Gamma/p^n \Gamma$ and $M/\mathfrak{m}_\Lambda M$ 
is finite.

\begin{definition}
 For such  a continuous $\Gamma$-module $M$, we define $H^1(\Gamma,M)$ by $M/(\gamma-1)M$ for $\gamma$ any topological generator of $\Gamma$.  It is independent of choice of the generator $\gamma$.
 \end{definition}
 
\it Remark. \rm $H^1(\Gamma,M)$ is also the continuous $H^1$.

 We have the following lemma :
 
 \begin{lemma}\label{generality}
  From an exact sequence of $\Gamma$-modules $$0 \rightarrow M_1 \rightarrow M \rightarrow M_2 \rightarrow 0$$ we get an exact sequence of abelian groups  $$0 \rightarrow M_1^{\Gamma} \rightarrow M^\Gamma \rightarrow M_2^\Gamma \rightarrow H^1(\Gamma,M_1).$$ 
 \end{lemma}
 
\begin{proof} It follows from the exact sequence of continuous cohomology
groups. This exact sequence is available as by  compactness the topology
of $M_1$ is induced by the topology of $M$ and the map $M\rightarrow M_2$
has a continuous section (in the category of sets). To get a continuous section apply Mittag-Leffler to 
the projective system of sections of $M/\mathfrak{m}^nM\rightarrow 
M_2 /\mathfrak{m}^nM_2$. \end{proof}
 
For $M$ as above, we denote by $H^1(\Gamma, M\otimes \Qp)$ the 
continuous cohomology where  $M\otimes \Qp$ carries the group topology
that induces on $M/\{ p^\infty-\mathrm{torsion}\}$ its topology. By compactness of $\Lambda$ and flatness
of $\Qp$ over $\Zp$, 
it is isomorphic to  $H^1(\Gamma, M)\otimes \Qp$.  

\begin{prop}\label{equivalent}
Leopoldt's conjecture is equivalent to the finiteness of $H^1(\Gamma,Y_\infty)=Y_\infty/(\gamma-1)Y_\infty$.  Leopoldt's conjecture is  also equivalent to the vanishing of $H^1(\Gamma,Y_\infty\otimes \Q_p)$.

\end{prop}

\begin{proof}
 Observe that $Y_\infty'/(\gamma-1)Y_\infty$ is the Galois group of the maximal abelian $p$ extension of $F$ unramified outside $p$.
\end{proof}

Thus,  via Proposition \ref{equivalent},  Leopoldt's conjecture is equivalent to  the finiteness of  $M_Q$'s of the proposition.

For later use we note:

\begin{cor}\label{needed}
 The $M_Q$'s for $Q=\{q_1,q_2\}$'s that are inert in $F_\infty/F$ span the finitely generated 
 $\Z_p$-module $Y_\infty/(\gamma-1)Y_\infty$.
\end{cor}



\section{Reciprocity and splitting  conjectures}\label{odd}

We now consider a finite set of primes $Q$ of $F$ away from $p$ 
and such that the image of  ${\rm Frob}_q$ for $q \in Q$ generates  $\Gamma$.  We let $m$ be the cardinality of $Q$.
 For each $n$ consider the  Sylow $p$-subgroup of the minus part of the ray class group of conductor $Q_n$, the ideal generated by the product of the primes above $\{q\}$ of $\f_n$. We denote this by  $\EA_{n,Q}^-$.

\begin{definition}\label{kay} 
  Let $\kay_{n,Q}^-$ denote the subgroup of the Sylow $p$-subgroup  of 
$(\Oc_{\f_n}/Q_n)^{*}$ on which complex conjugation $\in {\rm Gal}(\f_n/F)$ acts by $-1$, modulo the image of the $p$-power roots of unity $\mu_{p^{n+t}}$ of $\f_n$.  
\end{definition}

\begin{lemma}\label{kay1}
The group $\kay_{n,Q}^-$  is isomorphic as a ${\rm Gal}(\f_n/F)$-module  to $(\mu_{p^{n+t}})^{m}$ modulo the diagonally embedded $\mu_{p^{n+t}}$.

\end{lemma}

\begin{proof} 
If $q_n$ is a prime of $\f_n$ above $q$, 
the $p$-Sylow  of the multiplicative group $(k_{q_n})^*$ of the residue
field of $q_n$ is isomorphic by the reduction map  modulo $q_n$ to  $\mu_{p^{n+t}}$.
This follows from the fact that the primes in $Q$ are inert in $F_\infty/F$
and that $\mu_{p^{n+t}}$  is the  Sylow $p$-subgroup of the torsion subgroup of $\f_n^*$.
To conclude, note that these isomorphisms are compatible with the action of $\mathrm{Gal} (\f_n /F)$ if $q$ is inert in $\f$, and with the action of  $\mathrm{Gal} (\f_n  /\mathcal{F})$
if $q$ split in $\f$. 
\end{proof}

\begin{lemma}\label{injectivity}

We have the exact  sequence for each $n \geq 0$:
\begin{equation}\label{rub}
 0 \rightarrow  \kay_{n,Q} ^-\rightarrow \EA_{n,Q}^- \rightarrow \EA_n^- \rightarrow 0 
\end{equation}

We have the following commutative diagram where the vertical maps are induced by the inclusion  maps $\f_n \hookrightarrow \f_{n+1}$:

\ \ \ \ \ \ \ \ \ \  \ \ \ \ \xymatrix{
0 \ar[r]    &\kay_{n,Q} ^-    \ar[d] \ar[r]  &\EA_{n,Q}^- 
 \ar[d] \ar[r]           &\EA_{n}^- \ar[d] \ar[r]  &0\\
 0 \ar[r]   &\kay_{n+1,Q} ^-    \ar[r]  &\EA_{n+1,Q}^- 
   \ar[r]           &\EA_{n+1}^-  \ar[r]  &0\\
}

We also have the following commutative diagram where the vertical maps are induced by the norm maps $\f_{n+1} \rightarrow \f_n$:

\ \ \ \ \ \ \ \ \ \  \ \ \ \ \xymatrix{
0 \ar[r]    &\kay_{n+1,Q} ^-    \ar[d] \ar[r]  &\EA_{n+1,Q}^- 
   \ar[d] \ar[r]           &\EA_{n+1}^- \ar[d] \ar[r]  &0\\
 0 \ar[r]   &\kay_{n,Q} ^-    \ar[r]  &\EA_{n,Q}^- 
   \ar[r]           &\EA_{n}^-  \ar[r]  &0\\
}

All the vertical maps  in the first diagram are injective, and the first vertical map
of the second commutative diagram is surjective.
\end{lemma}

\begin{proof}

The horizontal exact sequences follow from:

--If ${\rm Cl}_{\f_n}$ and ${\rm Cl}_{\f_n,Q_n}$ denote the ray class group of conductor  1 and $Q_n$ of $\f_n$ respectively then we have an exact sequence
$$0 \rightarrow {(\Oc_{\f_n}/Q_n)^*}/\bar{E}_{\f_n} \rightarrow {\rm Cl}_{\f_n,Q_n} \rightarrow {\rm Cl}_{\f_n} \rightarrow 0$$ with $\bar{E}_{\f_n}$ the image of the global units $\Oc_{\f_n}^*$.

-- For $\epsilon\in \Oc_{\f_n}^*$, $\epsilon/\bar{\epsilon}$ is a root   of unity of $\f_n$.

-- $p>2$.

The commutativity of the diagrams is obvious.

Proposition 13.26 of \cite{Washington} proves  the injectivity of $\EA_{n}^- \rightarrow \EA_{n+1}^-$. The injectivity of the map  $\kay_{n,Q}^- \rightarrow \kay_{n+1,Q}^-$ follows by inspection. This in turn yields the injectivity of $\EA_{n,Q}^- \rightarrow \EA_{n+1,Q}^-$.

 Note that the norm map $\kay_{n+1,Q} ^- \rightarrow \kay_{n,Q} ^-$ is surjective as norm maps induce surjective maps between multiplicative groups of  finite extensions of finite fields.  
  
\end{proof}

\begin{cor}\label{added} Consider the exact sequence (\ref{rub}).

\begin{enumerate}

\item Taking direct limits of the  exact sequence (\ref{rub}) as $n$ varies, we get an exact sequence 
of discrete $\Lambda$-modules:
\begin{equation}\label{exact1}
0 \rightarrow \Q_p/\Z_p(1)^{m-1} \rightarrow  \EA_{\infty,Q}^- \rightarrow \EA_\infty^- \rightarrow 0.
\end{equation}

Note  further that as  the first non-zero term of the sequence (\ref{exact1})  is divisible, we have a $\Z_p$-linear section
$f:\EA_\infty^- \rightarrow \EA_{\infty,Q}^-$.  

\item Taking inverse limits of terms of the exact sequence (\ref{rub})  with respect to norm maps we   get the exact sequence  of  compact $\Gamma$-modules:
\begin{equation}
0 \rightarrow {\rm lim}_{\leftarrow}\kay_{n,Q}^- \simeq \Z_p(1)^{m-1} \rightarrow  {\rm lim}_{\leftarrow} \EA_{n, Q}^- \rightarrow  {\rm lim}_{\leftarrow} \EA_{n}^- \rightarrow 0.
\end{equation} which by class field theory is isomorphic to the exact sequence
$$ 0 \rightarrow I_Q  \rightarrow {\rm Gal}(\el_{\infty,Q}^-/\f_\infty) \rightarrow  {\rm Gal}(\el_{\infty}^-/\f_\infty) \rightarrow 0,
 $$ with  $\el_{\infty,Q}^-$ the maximal abelian  odd $p$-extension of $\f_\infty$ that is unramified outside the places above $Q$, and  $I_Q$    the subgroup of    ${\rm Gal}(\el_{\infty,Q}^-/\f_\infty) $
 generated by the inertia groups at the primes above $Q$ of $\f_\infty$. We set $\X_{\infty,Q}^-={\rm Gal}(\el_{\infty,Q}^-/\f_\infty)$ thus obtaining the exact sequence of $\Lambda$-modules
 \begin{equation}\label{exact} 0 \rightarrow \Z_p(1)^{m-1}  \rightarrow \X_{\infty,Q}^- \rightarrow  \X_\infty^- \rightarrow 0.\end{equation}
 
\end{enumerate}
\end{cor}

\begin{proof}
The  exactness in part  (2) follows using Mittag-Leffler criterion Prop. 2.7.3 of \cite{NSW}.  
 \end{proof}





\subsection{\it Cohomology classes:} We consider 
$m=2$ (recall that $m$ is the number of elements of $Q$).
The exact sequence (\ref{exact1}), gives rise to a cyclic   $\Z_p$-submodule of $H^1(\Gamma,{\rm Hom}(\EA_\infty^-,\Q_p/\Z_p(1)))=H^1(\Gamma,Y_\infty)$ the latter isomorphism by Iwasawa duality.  We  
  define a  cocycle corresponding to the above extension $c_\gamma \in {\rm Hom}(\EA_\infty^-,\Q_p/\Z_p(1)) $ by $c_\gamma= \gamma f -f$ 
  where $f$ is a  $\Z_p$-linear section  $\EA_\infty^- \rightarrow \EA_{\infty,Q}^-$ which we know exists  by Cor. \ref{added}.   The class of the cocycle $[c_\gamma]$ does not depend on the choice of the section $f$.  We can also obtain $[c_\gamma]$ as follows. From the exact sequence (\ref{exact1}), taking ${\rm Hom}_{\Z_p}(\EA_\infty^-, -)$, using the divisibility of $\Q_p/\Z_p(1)$ we deduce the exact sequence $$ 0 \rightarrow {\rm Hom}(\EA_\infty^-,\Q_p/\Z_p(1))
\rightarrow {\rm Hom}(\EA_\infty^-,\EA_{\infty,Q}^-) \rightarrow {\rm Hom}(\EA_\infty^-,\EA_{\infty}^-) \rightarrow 0,$$ and then taking $\Gamma$-invariants, Lemma \ref{generality} gives 
  $$ 0 \rightarrow {\rm Hom}(\EA_\infty^-,\Q_p/\Z_p(1))^\Gamma
\rightarrow {\rm Hom}(\EA_\infty^-,\EA_{\infty,Q}^-)^\Gamma \rightarrow {\rm Hom}(\EA_\infty^-,\EA_{\infty}^-)^\Gamma \rightarrow^\delta $$ $$H^1(\Gamma, {\rm Hom}(\EA_\infty^-,\Q_p/\Z_p(1))).$$ The cohomology class $[c_\gamma]$ is $\delta({\rm id})$. We see that (\ref{exact1}) splits as a sequence of $\Lambda$-modules if and only if $[c_\gamma]=0$.

   The $\Z_p$-module generated by  $[c_\gamma]$  in the cohomology group, we call $N_Q \subset H^1(\Gamma, {\rm Hom}_{\Z_p}(\EA_\infty^-,\Q_p/\Z_p(1)))=H^1(\Gamma,Y_\infty)$, the latter being induced by the Iwasawa isomorphism. 
 
 \vskip .5cm 
 
 \subsection{\it The reciprocity conjecture}\label{splitting}
 
  \begin{conj} \label{rec}(Reciprocity conjecture)
 Under the Iwasawa isomorphism,  $N_Q$ is mapped isomorphically  to $M_Q$ (both of them are pro-$p$ cyclic groups as $m=2$).
 \end{conj}

We view this as a reciprocity conjecture as we have the isomorphism (induced by the Iwasawa isomorphism) $$H^1(\Gamma,Y_\infty)=H^1(\Gamma,{\rm Hom}(\EA_\infty^-,\Q_p/\Z_p(1))),$$ natural
$\Z_p$-lines $M_Q$ and $N_Q$  on both sides associated to pairs of primes $(q_1,q_2)$  such that the image of their Frobenius generates $\Gamma$. The conjecture predicts that these lines are preserved under the Iwasawa isomorphism.


\noindent {\it Remark:} We may make a $\Q_p$ version of this conjecture. Namely we conjecture that the $\Q_p$-span of the class in $H^1(\Gamma,Y_\infty \otimes \Q_p)$ arising  from  the extension (\ref{exact})  made using the pairing $Y_\infty \times X_\infty^- \rightarrow \Q_p (1),$ is the same as $M_Q \otimes \Q_p$. (If we assume the Leopoldt conjecture this is simply  asserting that $0=0$!)

\subsection{\it Heuristic justification for the conjecture vis a vis generalised Jacobians} \label{heuristics}

We develop an analogy mentioned in the introduction  a little further by considering a direct analog of Conjecture \ref{rec} for function fields.  Assume that $X$ is a smooth projective  defined over a finite field $k$,  $P,Q \in X(k)$,  with  $J,J_{P,Q}$ as before.
 Let $\ell$ be a prime different from the characteristic of $k$.  Consider the exact sequences of $\Gamma=\hat \Z={\rm Gal}({\overline k}/k)$ modules
 $$ 0\rightarrow  \Z_\ell(1)  \rightarrow {\rm Ta}_\ell(J_{P,Q}) \rightarrow {\rm Ta}_\ell(J) \rightarrow 0,$$ 
 which splits as abelian groups.  It is easily seen that it splits up to isogeny as $\Gamma$-modules 
using the Weil  bounds on eigenvalues of Frobenius.  The corresponding fact for number fields
is unknown, and in analogy with function fields we conjecture it below.

Using the Weil pairing we get isomorphisms  $$H^1(\Gamma, {\rm Hom}_{\Z_\ell}(J({\overline k})[\ell^\infty],\Q_\ell(1)/\Z_\ell(1))) \simeq H^1(\Gamma, {\rm Hom}_{\Z_\ell}({\rm Ta}_\ell(J),\Z_\ell(1))) $$ $$ \simeq H^1(\Gamma,{\rm Ta}_\ell(J))=J(k)[\ell^\infty].$$  Then just as we did in a similar situation earlier we can form a  cyclic subgroup of $H^1(\Gamma,{\rm Ta}_\ell(J))=J(k)[\ell^\infty]$  which  arises from the extension classes arising from  the exact sequence above.   As N.~Fakhruddin explained to us,   in this case  one can indentify this extension class  with the projection of  $(P)-(Q)$ to the $\ell$-part of  $J(k)$, in perfect analogy with our Conjecture \ref{rec}. One may allow $K$  to be any field in the above considerations, by using  the Kummer  map
$\widehat{J(K)} \rightarrow H^1(G_K,{\rm Ta}_\ell(J))$ where $\widehat{J(K)}$ is the pro-$\ell$ completion of  $J(K)$, instead of  the isomorphism $H^1(\Gamma,{\rm Ta}_\ell(J))=J(k)[\ell^\infty]$
when $K$ is a finite field.


\subsection{ \it Splitting conjectures}

\subsubsection{Splitting of ramification away from $p$}

We make the following splitting conjecture motivated by analogy with generalised Jacobians over finite fields.

\begin{conj}\label{rational1}
 The exact sequence (\ref{exact}) $\otimes \Q_p$, i.e.,
 $$
0 \rightarrow \Q_p(1)^{m-1} \rightarrow  \X_{\infty, Q}^-  \otimes \Q_p \rightarrow \X_\infty^- \otimes \Q_p \rightarrow 0,$$ 
 of  $\Gamma$-modules splits.
\end{conj}

\subsubsection{ Splitting of ramification at $p$}

We make an analogous conjecture for splitting of  a certain (cyclotomic) part of the ramification at $p$.  We recall  the following results of Iwasawa.
  
 \begin{lemma}\label{iwasawa1} (Iwasawa)  Let $\wp'_1,\cdots,\wp'_{s'}$  be the  places above $p$ of $\f_\infty$, and $\wp_1,\cdots,\wp_s$ the places of $F$ above $p$. Denote by $\f_{\infty,i}=\cup \f_{n,i}$  the corresponding extension of   the completion of $F$ for $i=1,\cdots,s'$. Let $G_{\wp_j}$ be  decomposition subgroups of $G={\rm Gal}(\f_\infty/F)$ at the places $\wp_j$ of $F$.
  
  Let ${\mathcal  U}_i:={\rm lim}_{\leftarrow}{ U^1_{\f_{n,i}}}  ,$ where $U^1_{\f_{n,i}}$ are the principal units  in the completion  
 $\f_{n,i}$ and the inverse limit is with respect to the  norm maps.  Let ${\mathcal U}=\Pi_{i=1}^{s'} {\mathcal U}_i$, which is a $\Z_p[[G]]$-module.  Then we have an isomorphism of $\Z_p[[G]]$-modules   ${\mathcal U}
 \simeq  (\bigoplus_{j=1}^s {\rm  Ind}_{G_{\wp_j}}^G\Z_p(1)) \bigoplus \Z_p[[G]]^{[F:\Q]}$.
 
 \end{lemma}
 \begin{proof}
 This  follows  easily from   Theorem 11.2.4 of \cite{NSW}. 
 
 \end{proof}
 
 \begin{cor}\label{imp}
  
  Recall that $s$ is the number of places above $p$ of $F$.
  
   -- 1. The $\Z_p$-rank  of the group generated by  inertia groups at the places  above $p$ of $\f_\infty$  in the Galois group of the  maximal odd  abelian $p$-extension $N_\infty$   of $\f_\infty$  on which $\Gamma$ acts by  the $p$-adic cyclotomic character $\chi$  on the inertia above $p$,  say $I_p$, is  $[F:\Q]+s$. 

   --  2. The $\Z_p$-rank  of the group generated by  the  inertia groups at the places above $p$ of $\f_\infty$, in the Galois group that we denote by $\X_{\infty,p}^-$,  of the  maximal odd  abelian $p$-extension $N'_\infty$   of $\f_\infty$   that is unramified outside $p$, and  on which $\Gamma$ acts by  the $p$-adic cyclotomic character $\chi$  on the inertia above $p$,  say $I'_p$, is  $[F:\Q]+s-1$. 

 \end{cor}
 
 \begin{proof}
By class field theory,  for every $n$, the image of inertia above $p$ in the Galois group
 of the maximal abelian odd $p$-extension  of $\mathcal{F}_n$
 is isomorphic to $U^1_{\mathcal{F}_n}$.
 The first part of the corollary 
 then follows from the last lemma and the fact that the image of inertia above $p$ in ${\rm Gal}(N_\infty/\f_\infty)$ is isomorphic  theory to ${\mathcal U} /(\gamma'-\chi(\gamma'))$ where $\gamma'$ is a generator of ${\rm Gal}(\f_\infty/F)$. Similarly the image of inertia above $p$ in ${\rm Gal}(N'_\infty/\f_\infty)$ is isomorphic by class field theory to ${ {\mathcal U} /(\gamma'-\chi(\gamma')) \over \Z_p(1)}$.
 \end{proof}

Consider  $\X_{\infty,p}^-$, the Galois group of the  maximal odd abelian $p$-extension, denoted   $N'_\infty$  above, of $\f_\infty$ that is unramified outside $p$, and  on which ${\rm Gal} (\f_\infty/F)$ acts on the subgroup  $I'_p$ generated by the inertia groups at places above $p$ via the  $p$-adic cyclotomic character $\chi$.
Then we have an exact sequence of $\Lambda$-modules
\begin{equation}\label{atp}
0 \rightarrow  I_p' \otimes \Q_p  \rightarrow  \X_{\infty, p}^-  \otimes \Q_p \rightarrow \X_\infty^- \otimes \Q_p \rightarrow 0,
\end{equation}  We  know by Cor. \ref{imp} that $I'_p \otimes \Q_p$ is isomorphic to $\Q_p(1)^{[F:\Q]+s-1}$ as $\Lambda$-module. 

We make in the situation another splitting conjecture.

\begin{conj}\label{rational2}
The exact sequence (\ref{atp}) of $\Lambda$-modules splits.
\end{conj}



\section{ Relation to  Leopoldt's conjecure}

We show that the splitting conjectures are equivalent to Leopoldt's conjecture.

We begin with some generalities.  We denote by $F_p^*$ the group $\Pi_{v|p}F_v^*$, $U_{F}$ the group $\Pi_{v|p} U_{F_v}$  with $U_{F_v}$ the units of $F_v$. We denote by $U_F^1$ the group $\Pi_{v|p} U_{F_v}^1$ of 1-units. 


\begin{definition}\label{isogeny}

-- We say that a $\Lambda$ map $M \rightarrow N$ of compact finitely generated  torsion $\Lambda$-modules is an isogeny if the kernel and cokernel are torsion abelian groups (necessarily of bounded  exponent and finitely generated as $\Lambda$-modules).

-- If $$0 \rightarrow K \rightarrow M \rightarrow N \rightarrow 0$$ is a sequence of compact finitely generated  torsion $\Lambda$-modules, we say that it splits up to isogeny if  the sequence of $\Lambda$-modules $$0 \rightarrow K \otimes \Q_p \rightarrow M \otimes \Q_p \rightarrow N \otimes \Q_p \rightarrow 0$$ splits.
\end{definition}

We have a lemma that is  a direct consequence of the definition.
 
\begin{lemma}\label{useful1}
Consider an exact sequence  $$0 \rightarrow K \rightarrow M \rightarrow N \rightarrow 0$$  of compact finitely generated torsion $\Lambda$-modules. It   splits up to isogeny if and only if $M$ has a $\Lambda$-submodule $N'$  with  the natural map $N'  \rightarrow N$ an isogeny.
\end{lemma}

The  following lemma is easily proved.

\begin{lemma}\label{obvious}
 Conjecture \ref{rational1}, in the case $Q=\{q_1,q_2\}$ with $q_i$ inert in 
 $F_\infty/F$, is true if and only if  there is a $\Z_p$-extension $L_Q$  of $\f_\infty$ that is Galois over $F$, ramified  at $q_1,q_2$ and unramified everywhere else, and on which complex conjugation acts by $-1$.   
 \end{lemma}
  Note that  $\Gamma$ acts on ${\rm Gal}(L_Q/\f_\infty)$ by the $p$-adic cyclotomic character as the $q_i$ are inert in $F_\infty/F$, and thus $L_Q$ is a $\Z_p$-Kummer extension of $F$, with $L_Q/\f_\infty$ unramified outside  the primes above $Q$, and ramified at all the primes in $Q$.   Leopoldt's conjecture  predicts that there is a unique such extension.

\begin{proof} Only the ``only if''  direction needs proof. Assume that the conjecture is true.  Then
by the previous lemma  we get  $X\subset \mathcal{X}_{\infty, Q} ^-$ a $\Lambda$-submodule
with $X\rightarrow \mathcal{X}_{\infty} ^-$ having kernel and cokernel killed 
by a power of $p$. We define $L_Q $  as the subfield of $\mathcal{L}_{\infty, Q}^-$ which under the Galois correspondence  is such that it is Galois over $\mathcal{F}_\infty$, and  its Galois group over $\mathcal{F}_\infty$ is the quotient of $\mathcal{X}_{\infty, Q}^-/X$
by its $p$-power torsion.
\end{proof}

 \begin{theorem}\label{M_Q}
  Consider $Q=\{q_1,q_2\}$ a  tuple of primes of $F$, inert in $F_\infty/F$. Then the exact sequence in Conjecture \ref{rational1} splits if and only if the degree $ 0$  Frobenius submodule $M_Q$ of $Y_\infty/(\gamma-1)Y_\infty$ is  a finite group.
\end{theorem}

\begin{proof} 
  Consider the 1-units $U^1_F$ of $\Pi_{v|p} F_v^*$, and the subgroup $\overline{ E^1_F}$
  the closure of the global units $E_F^1$ that are 1 mod $v$ for all $v|p$.  We note the standard exact sequence from class field theory:
  \begin{equation}\label{cft}
  0 \rightarrow U^1_F/\overline{ E^1_F} \rightarrow Y'_\infty/(\gamma-1)Y_\infty \rightarrow C \rightarrow 0,
  \end{equation} where $Y'_\infty/(\gamma-1)Y_\infty$ is the Galois group of the maximal abelian $p$-extension of $F$ unramified outside $p$, 
  where $C$ is the Sylow $p$-subgroup of the ideal class group of $F$ (cf. Chapter 13 of \cite{Washington}).

 By  Lemma \ref{obvious} we  have to show that the existence of an $L_Q$ as in its statement is equivalent to $M_Q$ being finite. Recall that $\widehat{F_p^*}$ is the product 
$\prod_{v\mid p} \widehat{F^*_v}$ for $v$ the primes of $F$ above $p$ and 
we have a natural localisation map $\mathrm{loc}_p : \widehat{F^*}\rightarrow      
\widehat{F_p^*}$.
 By the results of \S \ref{kummertheory},  the existence of a $\Z_p$-Kummer extension $L_Q$ of $F$, such that $L_Q/\f_\infty$ is ramified precisely at all the primes above $Q$,  is equivalent to the existence of an element $\alpha$ of $\widehat{E_Q} \subset \widehat{F^*}$   such that $v_t(\alpha) \neq 0$ for $t=q_1,q_2$, and  $\mathrm{loc}_p (\alpha )$ is torsion. 
 By replacing $\alpha$ by a power of $\alpha$, we can suppose 
 that $\mathrm{loc}_p (\alpha )$ is trivial.


Suppose that the exact sequence of Conjecture 5.9. splits.
Then we get an $\alpha$ as above. Its image by the map
 $\widehat{E_Q}\rightarrow U^1_F/\overline{ E^1_F} \rightarrow 
 Y'_\infty/(\gamma-1)Y_\infty$ is $\mathrm{Frob}_{q_1}^{a_1}
\mathrm{Frob}_{q_2}^{a_2}$ for  $a_i \in \Z_p$.
It is trivial as $\mathrm{loc}_p (\alpha)$ is trivial. 
As  $v_t(\alpha) \neq 0$ for $t=q_1,q_2$, we get  that $a_i \neq 0$ and this produces a non-trivial $\Zp$-linear 
relation  between $\mathrm{Frob}_{q_1}$ and $\mathrm{Frob}_{q_2}$, hence
$M_Q$ is finite.

Conversely suppose that $M_Q$ is finite. Let, for $i=1,2$, $\alpha_i$ be elements
of $F^*$ which generates a power of the ideal $q_i$. As $M_Q$ is finite, the images 
of $\alpha_1$ and $\alpha_2 $ in $U^1_F/\overline{ E^1_F}$ are $\Zp$-linearly independent. It follows that 
there exists $a_1$ and $a_2$ non zero elements of $\Zp$ and $\alpha_3\in \overline{ E^1_F}$
such that $\alpha_1^{a_1} \alpha_2 ^{a_2} \alpha_3 =1$ in $U^1_F$. Lifting $\alpha_3$
to  $\epsilon \in \widehat{E_F^1}$, we get an element 
$\alpha :=  \alpha_1^{a_1} \alpha_2^{a_2}  \epsilon\in
\widehat{F^*}$. It satisfies the required properties : $v_{q_i}(\alpha)\not= 0$
and $\mathrm{loc}_p (\alpha )=1$. The theorem follows.

  
  \end{proof}

 \begin{cor}\label{implication1}
  Conjecture \ref{rational1} is true for all tuples of primes $Q=\{q_1,q_2\}$ which are inert in $F_\infty/F$ if and only if Leopoldt's conjecture is true.
 \end{cor}

 \begin{proof}
  We need only prove that the truth of Conjecture \ref{rational1} for tuples $Q=\{q_1,q_2\}$ inert in $F_\infty/F$ implies Leopoldt's conjecture. For this   we note (cf. Cor. \ref{needed})  that the $M_Q$'s span  the finitely generated $\Z_p$-module $Y_\infty/(\gamma-1)Y_\infty$ for such $Q$. By the  theorem, Conjecture \ref{rational1} implies that $M_Q$ is of finite order.
 \end {proof}
 
\it Remark. \rm  
The fact that Leopoldt's conjecture implies the splitting of the exact sequence of conjecture
\ref{rational1} also follows  as then the $\Lambda$-modules 
$ \Q_p(1)^{m-1}$ and  $X_\infty^- \otimes \Q_p $ have characteristic polynomials which 
are prime to each other.

\begin{prop}\label{splittingatp}
 
  Conjecture \ref{rational2} is equivalent to Leopoldt's conjecture.
\end{prop}

\begin{proof}
 Consider $E'_F$ the group of $p$-units of $F$. By the unit theorem it has $\Z$-rank $[F:\Q]+s-1$.
  We claim that (\ref{atp}) splits if and only if
   the $p^\infty$-Kummer extension 
  $\el=\f_\infty({E'_F}^{1/p^\infty})$ of $\f_\infty$, whose Galois group has $\Z_p$-rank $[F:\Q]+s-1$ (see the results of \S \ref{kummertheory} for instance), is   almost totally ramified at $p$. If Leopoldt's conjecture is true,  which is equivalent to $\X_\infty^-/(\gamma-\chi(\gamma))$ being finite, then  as the action of $\Gamma$ on ${\rm Gal}(\el/\f_\infty)$ is via the $p$-adic cyclotomic character, $\el$  is almost linearly disjoint from  $\el_\infty^-$ over $\f_\infty$, which implies that $\el/\f_\infty$ is almost totally ramified above $p$. On the other hand if $\el/\f_\infty$ is almost totally ramified at $p$, we  may deduce   that the $p$-adic  completion of the units $E_F$ of $F$  in $U_F^1$ has rank $[F:\Q]-1$. This is another form of the Leopoldt conjecture. For the  deduction we 
  use  Proposition \ref{ramification} and Lemma \ref{rami}. Namely we see that the $\Z_p$-rank of the subgroup generated by  the inertia groups above $p$
  in ${\rm Gal}(\el/\f_\infty)$  $=(\Z_p$-rank of the submodule $\overline{ E_F^1}$ of $U_F^1$ )$+s$.
\end{proof}

\section{Some evidence for the reciprocity conjecture }

Using the Kummer theory of \S \ref{kummertheory}, when (\ref{exact}) splits  after tensoring with $\Q_p$ as $\Lambda$-modules,  we measure precisely its failure to split over $\Z_p$. 
This then lends  support to our reciprocity conjecture. 
  
  We define $G$ to be the group  $Y'_\infty/(\gamma-1)Y_\infty$.
Hence the exact sequence  (\ref{cft}) becomes : 
  \begin{equation}\label{cftbis}
  0 \rightarrow U^1_F/\overline{ E^1_F} \rightarrow G
 \rightarrow C \rightarrow 0,
  \end{equation}
We define the quotient $G'$ of $G$
 by the image in $U_F^1/\overline{E_F^1}$ of the roots of unity, denoted by $\mu$, of $p$-power order of the product $F_p^*$ of the multiplicative groups of completions of $F$
at primes above $p$. 
We consider as before $Q=\{q_1,q_2\}$ with $q_1$ and $q_2$ distinct primes 
not above $p$ and  inert in $F_\infty/F$. 

\begin{theorem}\label{intersection}
  Assume that the order of the degree 0  Frobenius module $M_Q$ is finite. It is equivalent to assuming that a Kummer $\Z_p$-extension $L_Q/F$ exists with $L_Q/\f_\infty$ unramified at a place if and only if it does not lie above a prime in $Q$  (cf. Lemma \ref{obvious} and  Theorem  \ref{M_Q}).  Then  for any such   $L_Q$, the degree $[L_Q \cap \el_\infty^-:\f_\infty]$
  is divisible by the order $ m_Q$ of the image of  $M_Q$ in $G'$. Furthermore there exists 
  such an $L_Q$  with $[L_Q \cap \el_\infty^-:\f_\infty]=m_Q$.
\end{theorem}

Note that if Leopoldt's  conjecture is true for $F$
and $p$, then such an $L_Q$ exists and is unique.

\begin{proof} 
In the first part of the proof, let us fix $L_Q$ as in the statement of the theorem
and let us prove that $[L_Q \cap \el_\infty^-:\f_\infty]$
 is divisible by  $ m_Q$.

 By the  Kummer theory in \S \ref{kummertheory}, one gets an $\alpha \in \widehat{F^*}$, in fact even in the $p$-adic completion $\widehat{E_Q}$  of the $Q$-units of $F^*$,  such that $L_Q=F(\mu_{p^\infty},\alpha^{1 \over p^\infty})$.  We may assume that $\alpha \notin{( \widehat{F^*})}^p$
  equivalently not in $(\widehat{E_Q})^p$.  
  

  \begin{lemma}\label{intdeg} For each $n$, $F(\mu_{p^\infty},\alpha^{1 \over p^n})$ is 
 cyclic of order $p^n$ over $\f_\infty$.
 The valuations $v_{q_1}(\alpha)$ and $v_{q_2} (\alpha )$ are non zero and generate
 the same ideal ideal in $\Zp$.  If $(p^a)$ is this ideal,
 we have $p^a=[L_Q \cap \el_\infty^-:\f_\infty]$
 
 \end{lemma}
 
\begin{proof}
The first part of the lemma follows from  
$H^1({\rm Gal}(\f_\infty/F),\mu_{p})=0$. This is a consequence of 
 $\mu _p (F)=1$.

 Consider the map $\widehat{F^*} \rightarrow \widehat{\Q^*} \rightarrow 
\widehat{\Qp ^*}$, where the first arrow is induced by the  norm and the second
one by the localisation map $\mathrm{loc}_p$. It sends $\alpha$ to $N(q_1)^{v_{p_1}(\alpha )}
N(q_2 ) ^{v_{p_2}(\alpha)}$. Its image in $U^1_{\Zp}$ is trivial as $\mathrm{loc}_p (\alpha)$ is torsion.
As $q_1$ and $q_2$ are inert in $F_\infty $, the norms $N(q_1)$ and $N(q_2 )$
have images in $U^1_{\Zp}$ such that that 
$N(q_1)-1$ and $N(q_2)-1$ 
topologically generate the same ideal in $\Zp$. The second part 
of the lemma follows.

The third  part follows from the fact that $\mathcal{F}_n (\alpha^{1/p^a})$
is unramified at $q_i$ for $n+t\geq a$ if and only if $p^a$ divides $v_{q_i } (\alpha )$
and the first part of the lemma. 

\end{proof}

 By the Kummer theory in \S \ref{kummertheory} we deduce that $\alpha \in \widehat{E_Q}$ of the first paragraph of the proof  has the properties:
 
--  ${\rm loc}_p (\alpha)$ is torsion, and hence the natural norm map $\widehat{E_Q} \rightarrow \widehat{\Q_p^*}$ evaluates $\alpha$ to 1.
  
  -- $v_t(\alpha)=0$ for $t \notin Q$
  
  -- $v_{q_i}(\alpha) \neq 0$ and generate the same ideal say $(m)$ in $\Z_p$.
  
  -- $\alpha \notin (\widehat{F^*})^p$

 We note that  the $\Z_p$-submodule $M_Q$  of $G$ is generated by any element 
 of the form ${\rm Frob}_{q_1}^{a_1}{\rm  Frob}_{q_2}^{a_2}$ with $a_i \in \Z_p^*$, such that its image in ${\rm Gal}(F_\infty/F)$ is trivial. An explicit generator is gotten by taking $a_1= -{\rm log}_{\langle N(q_1)\rangle} \langle N(q_2) \rangle), a_2=1$ where by $\langle N(q_i) \rangle$ we mean the projection of $N(q_i)$ to $\Gamma$ in the decomposition $\Z_p^*= \Z/(p-1)\Z \times \Gamma$.
 
 
 Consider  the image of such an $\alpha$ in the Galois group $G'$
 by the map $ \widehat{E_Q}\rightarrow U_F ^1\rightarrow G \rightarrow G'$. 
 On the one hand it is trivial as ${\rm loc}_p(\alpha)$ is torsion. But on the other hand, 
 as $\mathrm{loc}_p (N(\alpha ))=1$,  it is  also of the form $({\rm Frob}_{q_1}^{a_1}{\rm Frob}_{q_2}^{a_2})^m$  with ${\rm Frob}_{q_i} $denoting the Frobenius at $q_i$ in the abelian Galois group $G'$, and with $a_i \in \Z_p^*$.  From this  we deduce that  $m_Q$ 
 divides $m$. By Lemma \ref{intdeg} we deduce that
 the degree $[L_Q \cap \el_\infty^-:\f_\infty]$
  is divisible by the order $ m_Q$ of the image of  $M_Q$ in $G'$. This finishes the first
  part of the proof.




The following lemma finishes the proof of the theorem.

  \begin{lemma}\label{val}
  There is  an element $\alpha \in \widehat{F^*}$  such that
  
  -- ${\rm loc}_p (\alpha)$ is torsion
  
  -- $v_t(\alpha)=0$ for $t \notin Q$
  

   -- $(v_{q_1}(\alpha))=(v_{q_2}(\alpha))=(m_Q)$ as ideals in $\Z_p$.
  
  We note for later use that if  $M_Q$ is trivial we get an element $\alpha \in \widehat{F^*}$  such that
  
  -- ${\rm loc}_p (\alpha)=1$
  
  -- $v_t(\alpha)=0$ for $t \notin Q$
  

   -- $(v_{q_1}(\alpha))=(v_{q_2}(\alpha))=\Z_p$.

  \end{lemma}

Consider $L_\alpha=\f_\infty(\alpha^{1 \over p^\infty})$ with $\alpha$ as in the first part of the lemma.   By  \S \ref{kummertheory} we get that $L_\alpha$ is a $\Z_p$-Kummer extension such that $L_\alpha/\f_\infty$ is ramified exactly at the  primes above $q_1,q_2$. Furthermore by  Lemma \ref{intdeg}, $[L_\alpha\cap \el_\infty^-:\f_\infty]=m_Q$.

  Thus we only need to prove the lemma.  We recall the  exact sequence (\ref{degree2}) from earlier:
  $$0 \rightarrow Y_\infty /(\gamma-1)Y_\infty  \rightarrow Y_\infty'/(\gamma-1)Y_\infty  \rightarrow \Z_p \rightarrow 0.$$

  Recall that $M_Q$ is a submodule of $Y_\infty /(\gamma-1)Y_\infty  $.
  Consider
  a generator $F_Q$  of $M_Q$ which we may write as ${\rm Frob}_{q_1}^{a_1}{\rm Frob}_{q_2}^{a_2}$ with $a_i \in \Z_p$.  We note again  that $a_i \in \Z_p^*$ by the assumption that the primes in $Q$ are inert in $F_\infty/F$. Let $n$ be the order of the prime to $p$ part of the class group of $F$.  Then we may regard  $(q_1^{a_1}q_2^{a_2})^{nm_Q}$ as a well-defined  element $\alpha'$
  of $\widehat{E_Q}/\widehat{E_F}$ as follows.   Choose $m$ large enough so that $q_i^{p^m}$ has image in the class group $Cl_F$ of $F$ of order prime to $p$.
  Choose $b_i \in \Z$ so that $a_i$ is congruent to $b_i$ modulo $p^m$: write $a_i=b_i+p^{m} c_i$ with $c_i \in \Z_p$.
  Note that $(q_1^{b_1}q_2^{b_2})^{{nm_Q}}$ has trivial image in  the  class group  $Cl_F$,  as 
  ${({\rm Frob}_{q_1}^{a_1}{\rm Frob}_{q_2}^{a_2})}^{{m_Q}}$ is trivial in $G'$, and thus gives rise to   a well-defined  element $\beta$ of   $E_Q/E_F$ whose image in $\widehat{E_Q}/\widehat{E_F}$ we denote by the same symbol.   Here we are using the exact sequence   (\ref{degree2}).  Furthermore
  $(q_1^{np^mc_1}q_2^{np^mc_2})^{m_Q}$ gives rise to   a well-defined  element $\beta'$   of $\widehat{E_Q}/\widehat{E_F}$. Thus taking product  $\beta\beta'$ we see that  altogether $(q_1^{a_1}q_2^{a_2})^{nm_Q}$ gives rise to  a well-defined  element  $\alpha'$ of $\widehat{E_Q}/\widehat{E_F}$ independent of choice of $m$.  Furthermore, the natural map
   $\widehat{E_Q}/\widehat{E_F} \rightarrow U_F^1/ \overline{E_F^1}\mu$ sends $\alpha'$ to 1.

  Choose $\alpha'' \in \widehat E_Q$ which projects to $\alpha'$, and by choice maps  to an element of $\overline{E_F^1}\mu$ under the natural map
  $\widehat{E_Q} \rightarrow  U_F^1$.  Thus the image of $\alpha''$ in $F_p^*/\mu$ is the image of  an $e'$ for $e' \in \overline{E_F^1}$. Let $e$ be any  inverse image of $e'$ under the natural map $\widehat{E_F} \rightarrow \overline{E_F^1}$. We set $\alpha=\alpha''.e^{-1}$, and see that  ${\rm loc}_p(\alpha)$ is torsion, $\alpha \in \widehat{E_Q}$, and $(v_{q_i}(\alpha))=({m_Q})$. The second part of the lemma follows by a similar argument.

\end{proof}

We may verify one consequence of our reciprocity conjecture as (ii) of the following corollary:

 \begin{cor}\label{evidence}

(i) The exact  sequence  (\ref{exact})  of $\Lambda$-modules splits if and only if $m_Q=1$.

(ii)  For a tuple of primes $Q=\{q_1,q_2\}$ inert in $F_\infty/F$,  $M_Q$ trivial implies that the exact sequence (\ref{exact1})  of $\Lambda$-modules splits.
\end{cor}

\begin{proof}

(i) The sequence (\ref{exact}) splits if and only if there is a Kummer $\Z_p$-extension
$L_Q$ as in the theorem with the property that $L_Q \cap \el_\infty^-$ is trivial.
This is equivalent by the theorem to $m_Q=1$.

(ii) By  the lemma \ref{val} in the proof above,  under the assumption that $M_Q$ is trivial  we get an element $\alpha$ of $\widehat{E_Q}$ such that ${\rm loc}_p(\alpha)=1$, and $v_{q_i}(\alpha)$
 is a unit for $q_i$ in  $Q$. Then for any $n$, the extension of $\f_n$ given by $\f_n(\alpha^{1 \over {p^{n+t}}})$ is cyclic of degree $p^{n+t}$, unramified outside the primes above $Q$, and has no non-trivial unramified subextension. By class field theory this provides a compatible sequence of splittings of the exact sequences (\ref{rub}), and thus a splitting of  (\ref{exact1}).

\end{proof}

\noindent{\it Remark:} We may also verify the converse of part (ii) of the corollary  in some situations, for instance when $\f_\infty/F$  has a unique prime above $p$ and is totally ramified at this prime.

\section{Even extensions of Iwasawa modules}\label{even}


  We state  the theorem of Iwasawa proved in U4 of \cite{Iwasawa}.
 
 \begin{theorem}\label{il}(Iwasawa)
 Leopoldt's conjecture  is  equivalent to  the following statement:   
 For any set of finite places $Q$ disjoint from $S_p$ the map
 $$H^1(S _p \cup Q,\Q_p/\Z_p) \rightarrow \prod_{v \in Q}{H^1(I_v,\Q_p/\Z_p)^{D_v}}$$ is surjective.
\end{theorem}

\noindent{\it Remark:}  Iwasawa stated his criterion as:  Leopoldt's conjecture, cf. Conjecture \ref{Leopoldt},  is true if and only for every prime $q$  prime to $p$ of $F$,  the image of inertia at  the prime $q$  in  ${\rm Gal}(F_{p,q}/F)$, with $F_{p,q}$ the maximal abelian $p$-extension of $F$ unramified outside $p,q$,   has order $e(q)$, the $p$-part of the order of the multiplicative group of the residue field at $q$, denoted by $k_q^*$.

Now we transcribe the result of Iwasawa  into an Iwasawa theoretic setting, i.e.,  a statement over $\f_\infty$.  
 It   stands in counterpoint to the situation in the odd case.

Consider a finite set of primes $Q$  away from $p$  of $F$ such that their norm is 1 modulo $p$.
(if $v\in Q$  is such that $p$ does not 
divide $N(q)-1$, $H^1(I_v,\Q_p/\Z_p)^{D_v}$ is trivial).
We consider the maximal abelian $p$-extension $M_\infty(Q)$  of $F_\infty$ that is unramified outside $p$ and $Q$ with Galois group $Y_{\infty,Q}$. We assume for simplicity that $Q$ contains only  one place $q$ 

Then we have an exact  sequence of $\Gamma$ or $\Lambda$-modules:
\begin{equation}\label{exact2}
0 \rightarrow K_Q \rightarrow Y_{\infty,Q} \rightarrow  Y_\infty \rightarrow 0
\end{equation}
where the Iwasawa module $K_Q$ is simply given by $\Lambda/((1+T)^{p^b}-u^{p^b})$ where
$\gamma(\zeta_{p^n})=\zeta_{p^n}^u$ and $u^{p^b}-1$  is divisible by the same power of $p$ as 
$N(q)-1$.  One sees this as in \cite{Greenberg}  using Kummer theory, which also shows that the exact sequence (\ref{exact2}) splits up to isogeny.

\begin{lemma}
 Leopoldt's conjecture is true for $F,p$ if and only if the exact sequence (\ref{exact2}) remains exact on going modulo $T$ for each choice of $q$.
\end{lemma}

\begin{proof}
 Note that the sequence (\ref{exact2})  remains exact  on going modulo $T$ if and only if the image of an inertia group above $q$ in  ${\rm Gal}(F_{p,q}/F)$ 
 is of order the $p$-part of $N(q)-1$, namely $e(q)$. Then we are done by the equivalence of Theorem \ref{il}.
\end{proof}

It is interesting to note that in the odd case  the sequence  (\ref{exact}) remains exact on going modulo $T$, while
its splitting up to isogeny  (for all $Q$) is equivalent to Leopoldt's conjecture. In the even case, the exact sequence
(\ref{exact2}) does split up to isogeny,  as shown  by Greeenberg in loc. cit. using Kummer  theory, but its remaining exact on going modulo $T$ is equivalent to Leopoldt's conjecture. Iwasawa's criterion, and the one in this paper, are dual in a sense that gains precision using considerations in the next section.

\section{ Our criterion using Poitou-Tate}

The criteria for Leopoldt's conjecture contained in Cor. \ref{implication1} and Prop. \ref{splittingatp}  above may also be derived from formulas  of Greenberg and Wiles in Galois cohomology, which are consequences of  the Poitou-Tate exact sequence. We apply these formulas to $\Q_p$-representations where they are also valid.

For a finite dimensional vector space $V$ over $\Q_p$  endowed with a continuous action 
of $G_F$ that is everywhere almost unramified, and a set of Selmer conditions $\mathcal L=\{\mathcal L_v\}$ for $\mathcal L_v \subset H^1(F_v,V)$  with $\mathcal L_v$ almost everywhere the unramified subgroup we have the formula

$$h^1_{\mathcal L}(F,V)-h^1_{\mathcal L^\perp}(F,V^*(1))=h^0(F,V)-h^0(F,V^*(1))+\sum_v ({\rm dim}_{\Q_p}{\mathcal L}_v - h^0(F_v,V)).$$

We apply this formula for $V=\Q_p(1)$, and with the Selmer conditions $\mathcal L$ to be unramified everywhere, in particular trivial at places above $p$.  The Leopoldt conjecture is equivalent
to $h^1_{\mathcal L^\perp}(F,\Q_p)$, which is at least 1-dimensional,  having dimension 1. 
Let $\delta:= h^1_{\mathcal L^\perp}(F,\Q_p)-1$ the defect to Leopoldt conjecture.
We easily get the criterion for Leopoldt conjecture contained in Cor. \ref{implication1}  by choosing $Q=\{q_1,q_2\}$ so that $h^1_{\mathcal L_Q^\perp}(F,Q_p)=h^1_{\mathcal L^\perp}(F,Q_p)-2$ if 
$\delta\geq 1$, and by applying again the formula of Greenberg 
and Wiles replacing $\mathcal{L}$ by $\mathcal{L}_Q$.  Here ${\mathcal L}_Q$ are the Selmer conditions  that arise when we allow ramification at $Q$. We get that $h^1_{\mathcal L}(F,\Q (1))= 
h^1_{\mathcal L_Q}(F,\Q (1))$, which contradicts the existence of a 
$\Zp$-extension as in lemma \ref{obvious}.
To see the criterion for the Leopoldt conjecture contained in  Prop. \ref{splittingatp} we relax the Selmer conditions by allowing ramification at $p$.

This method allows us to get a refinement of Prop. \ref{splittingatp} as follows. Consider the Galois group $M_{\infty,p}$ of the maximal odd, abelian $p$-extension of $\f_\infty$ which is unramified outside $p$. Then we have an exact sequence
$$0 \rightarrow T_p \otimes \Q_p \rightarrow M_{\infty,p}\otimes{\Qp} \rightarrow X_{\infty}^- \otimes \Q_p \rightarrow 0.$$ Here $T_p$ is the group generated by inertia at primes above $p$.  Consider any  odd character $\psi:{\rm Gal}(\f_\infty/F) \rightarrow \aQp^*$,  and assume 
that  the Selmer group $H^1_{\mathcal L}(G_F,\chi\psi^{-1})=0$ with the Selmer conditions being that of unramified everywhere. The vanishing would  follow from  Greenberg's conjecture that the Sylow $p$-subgroup of  the class group of $F_\infty$ is finite  as $\chi\psi^{-1}$ is an even character.  (For the case $\psi=\chi$ which corresponds to Prop. \ref{splittingatp} this just follows from finiteness of class numbers of number fields.) The character $\psi$ gives rise   to a prime ideal $P_\psi$  of height one of $\Lambda_\Oc$ where $\Lambda \otimes_{\Z_p} \Oc$ for  the ring of integers $\Oc$ of a $p$-adic field that contains values of $\psi$. Then  again using the Greenberg-Wiles formula one sees that the localisation of the above exact sequence (tensored with $\Lambda_\Oc$)  at $P_\psi$ is split as $(\Lambda_\Oc)_{(P_\psi)}$-modules  if and only if $(X_\infty^-)_{(P_\psi)}=0$. Thus every eigenspace of $X_\infty^- \otimes \Q_p$, assuming Greenberg's conjecture, intertwines  with ramification at $p$.

\section{Appendix}
 P. Colmez  showed us a  nice  argument using $L$-functions which,  assuming $F/\Q$ is a totally real finite Galois extension of $\Q$, proves that if Leopoldt's conjecture is false then $\zeta_{F,p}(s)$ has to vanish at $s=1$ where $\zeta_{F,p}(s)$ is the Deligne-Ribet $p$-adic $L$-function of $F$.
We note that by \cite{Serre} and \cite{Colmez},  the Leopoldt conjecture is true if and only if   $\zeta_{F,p}(s)$  has a pole at $s=1$.

We give Colmez's argument.  The $p$-adic zeta function  $\zeta_{F,p}(s)$ has a factorisation into certain $p$-adic Artin L-functions (cf. \cite{Artin}) $$\zeta_{F,p}(s)=\Pi_{\chi} L_{\Q,p}(s,\chi)^{\chi(1)},$$ with $\chi$ running through the irreducible $p$-adic  representations of ${\rm Gal}(F/\Q)$.    Note that  for $\chi$ a non-trivial representation, $L_{F,p}(s,\chi)$ is entire by the $p$-adic form of the Artin conjecture which  is proved in \cite{Artin} to  follow from the main conjecture. The factor for $\chi$ the trivial representation has a simple pole at $s=1$, and for other non-trivial abelian characters $\chi$, the corresponding factor does not vanish at $s=1$, by the known case of the Leopoldt conjecture for abelian extensions of $\Q$ (cf. \cite{Brumer}). Thus if the Leopoldt conjecture is false  for $F,p$,  for a representation $\chi$ of dimension at least 2, $L_{\Q,p}(1,\chi)$ vanishes and this by the factorisation formula forces $\zeta_{F,p}(1)$ to vanish.

We now  give a simple algebraic argument,  to deduce from the known cases of the Leopoldt  conjecture for abelian extensions of $\Q$,   that the Leopoldt defect $\delta_{F,p}$  can never be 1 for $F/\Q$ a totally real finite  Galois extension.  This could  well be known to the experts.  It  is an apparent strengthening of Colmez's result as it could conceivably happen that $\delta_{F,p}=1$ while $\zeta_{F,p}(1)=0$ (``non-semisimplicty of Leopoldt zeros'').  It will be nice to remove the assumption that $F/\Q$ is Galois;  this will require other methods.

\begin{prop}
 For $F/\Q$ a totally real  finite Galois extension and a prime $p$, the Leopoldt defect
 $\delta_{F,p}$ is never 1.
\end{prop}
\begin{proof} Suppose that  $\delta_{F,p}=1$. Let us call $N$ the compositum
of the $\Zp$-extensions of F. Let $L= \mathrm{Gal}(N/F)$ : $L$ is a free 
$\Zp$-module of rank 2.
The inclusion $F_\infty \subset N$ gives a surjective morphism $L\rightarrow \Gamma$.
 Let $H_1$ be the Galois group of $F/\Q$. We see that the action of $ H_1$
 on $\Qp\otimes L$ factors though the character $\eta$ giving the action 
 of $H_1$ on $\Qp \otimes (L/\Gamma)$. Hence the action of $H_1$ on $L$ factors through 
 $\eta$. Let $H$ be the kernel of $\eta$ and $F_\eta$ the field corresponding to $H$.
 The next lemma implies the existence of an extension $N'$ of $F_\eta$
 of Galois group a free $\Zp$-module of rank 2 such that $N=N'F$. This is impossible as $F_\eta$
 is abelian over $\Q$ and we know Leopoldt conjecture for $F_\eta$ and $p$.
 
 \begin{lemma} Let  
$ 1\rightarrow L\rightarrow G\rightarrow H\rightarrow 1$
be an exact sequence of profinite groups with $L $ a free $\Zp$-module  
of finite rank $d$. We suppose that $H$ is finite and acts trivially on $L$.
Then, there exists a free $\Zp$-module $L'$  of rank $d$ and a surjective 
morphism $G\rightarrow L'$.
 \end{lemma}
 
Let $o\in  H^2 (H, L)$ be the cohomology class defined by the extension.
Let us prove first that the conclusion of the lemma is equivalent to that there
exists  an inclusion $L\hookrightarrow L'$ of $\Zp$-modules of rank $d$  
such that the image $o'$ of $o$ in $H^2 (H, L')$ is trivial. 

Indeed, if there exists $L\hookrightarrow L'$  such that $o'$ is trivial, the 
pushout exact sequence $ 1\rightarrow L'\rightarrow G'\rightarrow H\rightarrow 1$ has 
a trivialisation $G'\rightarrow L'$. If we compose it with the morphism $G\rightarrow G'$
we get a morphism $G\rightarrow L'$ that coincide on $L$ with the inclusion 
$L\hookrightarrow L'$.

Conversely, if we have a surjection $G\rightarrow L'$, its restriction to
$L$ has a finite index image as $H$ is finite. This implies that 
this restriction is injective.  The morphism $G\rightarrow L'$
extends to the pushout $G'$ as a trivialisation of the pushout exact sequence.

Let us prove the lemma when $H$ is a $p$-group.
Let us prove it in this case by induction on the cardinality of $H$.
If $H$ is trivial, there is nothing to prove.
Otherwise, let $H'\subset H$
be a central subgroup of order $p$. Let $G''$ be the inverse image 
of $H'$ in $G$. As the action of $H$ on $L$ is trivial and 
$H'$ is cyclic, the group $G''$ is abelian.  If $G''$ has no torsion, 
we can apply the induction hypothesis to the exact sequence 
$1 \rightarrow G''\rightarrow G\rightarrow H/H'\rightarrow 1$
to get a surjective morphism of $G$ in $(\Zp )^d$. 
If $G''$ has torsion $T$, $T$ is cyclic of order $p$ and 
$G\rightarrow H$ induces an isomorphism of $T$ to $H'$.
We apply the induction hypothesis to the exact sequence 
$1 \rightarrow L\rightarrow G/T\rightarrow H/H'\rightarrow 1$.
We get a surjective morphism of  $G/T $ to $(\Zp ) ^d$
hence a surjective morphism of $G$ to $(\Zp ) ^d$.

Let us prove the lemma in the general case. Let $H_p$ a $p$-Sylow 
of $H$. Let $L\hookrightarrow L'$ be such that the image of the restriction of 
$o'$ to $H_p$ vanishes. The morphism $H^2 (H, L')\rightarrow H^2 (H_p , L')$
is injective. This follows from the injectivity of the maps 
$H^2 (H, L'/p^n L')\rightarrow H^2 (H_p , L'/p^n L')$ and Mittag-Leffler.
We see that $o'$ is trivial and this proves the lemma. 

\end{proof}

\nocite{*}
\bibliographystyle{plain}

\begin{thebibliography}{10}

\bibitem{Coates}
\newblock John Coates.
\newblock K-theory and Iwasawa's analogue
of the Jacobian.
\newblock Algebraic K-theory II, SLN 341, p. 502-520.

\bibitem{Brumer} Armand Brumer.
\newblock On the units of algebraic number fields. 
\newblock Mathematika 14, 1967, 121Ð124. 


\bibitem{Colmez} Pierre Colmez.
\newblock R\'esidu en s=1 des fonctions zeta p-adiques. 
\newblock  Invent. Math. 91 (1988), no. 2, 371Ð389.

\bibitem{Greenberg} Ralph Greenberg.
\newblock On $p$-adic $L$-functions and Cyclotomic fields II.
\newblock Nagoya Math. J., 67, 1977, p. 139-158.



\bibitem{Artin} Ralph Greenberg.
\newblock On p-adic Artin L-functions. 
\newblock Nagoya Math. J. 89 (1983), 77--87.




\bibitem{Iwasawa} Kenkichi Iwasawa.
\newblock Collected Papers I and II.
\newblock Springer-Verlag.


\bibitem{NSW} J. Neukirch, A. Schmidt, K. Wingberg.
\newblock Cohomology of number fields.
\newblock Springer-Verlag.

\bibitem{Serre} Jean-Pierre Serre.
\newblock Sur le r\'esidu de la fonction zeta p-adique d'un corps de nombres. 
\newblock C. R. Acad. Sci. Paris SŽr. A-B 287 (1978), no. 4, A183ÐA188. 


\bibitem{Washington}  Larry Washington.
\newblock Introduction to Cyclotomic Fields (2nd edition).
\newblock Springer-Verlag.


\bibitem{Wiles} Andrew Wiles. 
\newblock The Iwasawa conjecture for totally real fields. 
\newblock Ann. of Math. (2) 131 (1990), no. 3, 493--540. 

\end{thebibliography}

\end{document}